\documentclass[11pt]{article}
\usepackage{amsmath}
\usepackage[affil-it]{authblk}

\usepackage{lineno,hyperref}
\modulolinenumbers[5]

\usepackage{mathtools,amsfonts,amssymb,amscd,xspace,bbm,mathabx,enumitem,empheq,xparse,amsthm}

\theoremstyle{plain}
\newtheorem{example}{Example}
\newtheorem{theorem}{Theorem}
\newtheorem{corollary}[theorem]{Corollary}
\newtheorem{lemma}[theorem]{Lemma}
\newtheorem{proposition}[theorem]{Proposition}

\theoremstyle{definition}
\newtheorem{definition}[theorem]{Definition}
\theoremstyle{remark}
\newtheorem{remark}[theorem]{Remark}

\usepackage[margin=2.5cm]{geometry}
\usepackage{setspace}
\DeclarePairedDelimiter{\norm}{\lVert}{\rVert}
\newcommand{\half}{\frac{1}{2}}
\newcommand{\Rd}{\mathbb{R}^d}
\newcommand{\RN}{\mathbb{R}^N}
\newcommand{\RM}{\mathbb{R}^M}
\newcommand{\RK}{\mathbb{R}^K}
\newcommand{\RdxRd}{\Rd \times \Rd}
\newcommand{\ip}[2]{\langle #1\ , #2 \,\rangle}

\newcommand{\Sum}[2]{\sum_{#1=1}^{#2}}

\newcommand{\Tr}{\textrm{Tr}}

\newcommand{\PF}{\textrm{PF}(\Rd)}

\newcommand{\PFAB}{\textrm{PF}(A,B,\Rd)}

\newcommand{\supp}{\textrm{supp}}
\newcommand{\Span}{\textrm{span}}
\newcommand{\indicator}[1]{\mathbbm{1}_{\left[ {#1} \right] }}
\newcommand{\Vor}{\textrm{Vor}}

\newcommand{\R}{\mathbb{R}}
\newcommand{\N}{\mathbb{N}}
\newcommand{\LL}[2]{L^2(#1,#2)}

\newcommand{\ixset}[1]{\{1,\cdots, #1\}}

\allowdisplaybreaks

\begin{document}

%\begin{frontmatter}

\title{Duality and Geodesics for Probabilistic Frames}

\author{Clare Wickman}
\affil{Johns Hopkins University Applied Physics Laboratory}
%\ead{cwickman@math.umd.edu }

\author{Kasso Okoudjou}
\affil{Department of Mathematics, University of Maryland, College Park}
%\ead{kasso@math.umd.edu }
\date{\vspace{-5ex}}

\maketitle

%\begin{keywords}

\begin{abstract} Probabilistic frames are a generalization of finite frames into the Wasserstein space of probability measures with finite second moment. We introduce new probabilistic definitions of duality, analysis, and synthesis and investigate their properties. In particular, we formulate a theory of transport duals for probabilistic frames and prove certain properties of this class. We also investigate paths of probabilistic frames, identifying conditions under which geodesic paths between two such measures are themselves probabilistic frames. In the discrete case this is related to ranks of convex combinations of matrices, while in the continuous case this is related to the continuity of the optimal transport plan.
\end{abstract}

%\end{frontmatter}
\section{Introduction}
\subsection{Probabilistic frames in the Wasserstein metric}
Frames are redundant spanning sets of vectors or functions that can be used to represent signals in a faithful but nonunique way and that provide an intuitive framework for describing and solving problems in coding theory and sparse representation. We refer to \cite{Christensen2003, CasKut2013, OkoudjouFiniteFrame} for more details on frames and their applications. To set the notations for this paper, we recall that a set of column vectors $\Phi=\{\varphi_i\}_{i=1}^N\subset\Rd$ is a frame if and only if there exist  $0<A\leq B <\infty$ such that $$ \forall x\in\R^d,\quad A\norm{x}^2\leq\sum\limits_{i=1}^N\ip{x}{\varphi_i}^2\leq B\norm{x}^2.$$  Throughout this paper we abuse notation by also using $\Phi$ to denote $[\varphi_1\dots\varphi_N]^{\top},$  the analysis operator of the frame. The (optimal) bounds in the above inequality are the smallest and largest  eigenvalues of the frame operator $S_{\Phi}=\Phi^\top\Phi.$

Continuous frames are natural generalization of frames and were introduced by Ali, Antoine, and Gazeau \cite{AAG} (see also,\ \cite{FornasierRauhut}).   Specifically,  let $X$ be a metrizable, locally compact space and $\nu$ be positive, inner regular Borel measure for $X$ supported on all of $X$. Let $H$ be a Hilbert space.  Then a set of vectors $\left\{\eta_x^i, i\in\ixset{n},x\in X\right\}\subset H$ is a rank-$n$ (continuous) frame if, for each $x\in X$, the vectors $\left\{\eta_x^i, i\in\ixset{n}\right\}$ are linearly independent, and if there exist  $0<A\leq B < \infty$ such that $\forall f\in H$, $$A\norm{f}^2\leq \sum\limits_{i=1}^n\int_X |\ip{\eta^i_x}{f}|^2d\nu(x)\leq B\norm{f}^2.$$

In this paper, we are concerned with a different generalization of frames called probabilistic frames. Developed in a series of papers \cite{MartinRTF,KassoMartinOverview, MartinKassoPframe}, probabilistic frames are an intuitive way to generalize finite frames to the space of probability measures with finite second moment. The probabilistic setting is particularly compelling, given recent interest in probabilistic approaches to optimal coding, such as \cite{HanLengHuang,PowellWhitehouse}. In the new setting, the defining characteristics of a frame amount to a restriction on the mean and covariance matrix of the probability measure. Because of this characterization, a natural space to explore probabilistic frames is the Wasserstein space of probability measures with finite second moment, a metric space with distance defined by an optimal transport problem.

Before we give the definitions and the concepts needed to state our results we first observe that in the simplest example, each finite frame can be associated with a probabilistic frame.  Indeed, let  $\Phi=\{\varphi_i\}_{i=1}^N$ be a frame and  let $\{\alpha_i\}_{i=1}^N \subset (0, 1) $ be such that  $\sum_{i=1}^N \alpha_i =1$.  Then the canonical $\alpha$-weighted probabilistic frame associated with  $\Phi$ is the probability measure $\mu_{\Phi,\alpha}$ given by $$d\mu_{\Phi,\alpha}(x)=\sum_{i=1}^N \alpha_i \delta_{\varphi_i}(x).$$ More generally, a probabilistic frame $\mu$ for $\Rd$ is a probability measure on $\Rd$ for which there exist $0<A\leq B<\infty$ such that for all $x\in\Rd$, $$A\norm{x}^2\leq\int_{\Rd}\ip{x}{y}^2d\mu(y)\leq B\norm{x}^2.$$ Tight (finite, continuous, or probabilistic) frames are those for which the frame bounds are equal. While the work of this paper is limited to probabilistic frames on $\Rd,$ of interest is also the possible extension of these ideas to probabilistic frames on infinite dimensional spaces, as outlined in \cite{JorgensenSong2016}.

Probabilistic frames form a subclass of the continuous frames defined above. Indeed, defining the support of a probability measure $\mu$ on $\Rd$ as the set:
$$\supp(\mu):=\left\{x\in \Rd \textrm{ s.t. for all open sets } U_x \textrm{ containing } x, \mu(U_x)>0\right\},$$ it is not difficult to prove that the support of any probabilistic frame is canonically associated with a rank-one continuous frame. And conversely, certain continuous frames can be rewritten as probabilistic frames.  However, despite this equivalence, there is a strong difference in the tools available in the different settings.

We shall investigate probabilistic frames in the setting of the Wassertein metric defined on $P_2(\Rd),$  the set of probability measures $\mu$ on $\Rd$ with finite second moment:
$$M_2^2(\mu):=\int_{\Rd}\norm{x}^2d\mu (x) < \infty.$$

By \cite[Theorem 5]{KassoMartinOverview}, $\mu$ is a probabilistic frame if and only if it has finite second moment and the linear span of its support is $\Rd$.  This characterization can be restated in terms of the probabilistic frame operator for $\mu$, $S_{\mu},$ which for all $y\in\Rd$ satisfies:
$$S_{\mu}y=\int_{\Rd}\ip{x}{y}\,x\,d\mu(x).$$

Equating $S_{\mu}$ with its matrix representation $\int_{\Rd}xx^{\top}d\mu(x)$, the requirement that the support of $\mu$ span $\Rd$ is equivalent to this matrix being positive definite. 

The ($2$-)Wasserstein distance, $W_2$ between two probability measures $\mu$ and $\nu$ in $P_2(\Rd)$ is: $$W_2^2(\mu,\nu):=\inf_{\gamma}\left\{\iint_{\RdxRd}\norm{x-y}^2d\gamma(x,y): \gamma \in \Gamma(\mu,\nu)\right\},$$ 
where $\Gamma(\mu,\nu)$ is the set of all joint probability measures $\gamma$ on $\RdxRd$ such that for all $A,B \subset \mathcal{B}(\Rd)$, $\gamma(A\times\Rd)=\mu(A)$ and $\gamma(\Rd\times B)=\nu(B)$. The Monge-Kantorovich optimal transport problem is the search for the set of joint measures which induce the infimum; any such joint distribution is called an optimal transport plan. A special subclass of transport plans are those given by deterministic transport maps (or deterministic couplings), where $\nu$ can be written as the pushforward of $\mu$ by a map $T,$ denoted $\nu=T_{\#}\mu.$ That is, for all $\nu$-integrable functions $\phi$, $$\int_{\Rd}\phi(y)d\nu(y)=\int_{\Rd}\phi(T(x))d\mu(x).$$
\noindent When $\mu$ is absolutely continuous with respect to Lebesgue measure \cite[p.~150]{AGS2005}, then $$W_2^2(\mu,\nu):=\inf_{T}\left\{\int_{\Rd}\norm{x-T(x)}^2d\mu(x): T_{\#}\mu=\nu\right\}.$$ 
\noindent Equipped with the 2-Wasserstein distance, $P_2(\Rd,W_2)$ is a complete, separable metric space.  Convergence in $P_2(\Rd)$ is the usual weak convergence of probability measures, combined with convergence of the second moments.

A few structural statements can be made about probabilistic frames as a subset of $P_2(\Rd).$ For brevity, the probabilistic frames for $\Rd$ are denoted by $\PF,$ and $\PFAB$ denotes the set of probabilistic frames in $\PF$ with optimal upper frame bound less than or equal to $B$ and optimal lower frame bound greater than or equal to $A,$ with $0<A,B<\infty.$  Then $\PFAB$ is a nonempty, convex, closed subset of $P_2(\Rd)$.  The nonemptiness  and convexity are trivial to show.
With respect to closedness, let $\{\mu_n\}$ be a sequence in $\PFAB$ converging to $\mu\in P_2(\Rd)$. 
%Since $\int_{\Rd}\ip{x}{y}^2d\mu(x)$ is a continuous function of $y\in\Rd$, we can define 

Let $$y_0=\textrm{argmin}_{y\in S^{d-1}}\int_{\Rd}\ip{x}{y}^2d\mu(x).$$ Because $$\ip{x}{y_0}^2\leq\norm{x}^2\norm{y_0}^2\leq\norm{y_0}^2(1+\norm{x}^2),$$ it follows by definition of weak convergence in $P_2(\Rd)$ that $$\int_{\Rd}\ip{x}{y_0}^2d\mu_n(x)\rightarrow\int_{\Rd}\ip{x}{y_0}^2d\mu(x).$$  Since for all $n$, the values of $\int_{\Rd}\ip{x}{y_0}^2d\mu_n(x)$ are bounded above and below by $B$ and $A$, respectively, $\mu$ is an element of $\PFAB$.  Taking $A=B$, this also shows the closedness of $PF(A, \Rd)=PF(A,A,\Rd),$ the set of tight probabilistic frames with frame bound $A.$ However, the set of probabilistic frames itself is not closed, since one can construct a sequence of probabilistic frames whose lower frame bounds converge to zero: for example, a sequence of zero-mean, Gaussian measures with covariances $\frac{1}{n}I,\space n\in \N$.

\subsection{Our contributions}
The goal of this paper is to investigate two main topics on probabilistic frames in the setting of the Wasserstein space. The first topic is the notion of duality. For a finite frame, $\Phi=\{\varphi_i\}_{i=1}^N \subset \Rd,$ a set $\Psi=\{\psi_i\}_{i=1}^N \subset \Rd$ is said to be a dual frame to it if for every $x\in\Rd,$ $$x=\sum\limits_{i=1}^N\ip{x}{\varphi_i}\psi_i.$$

It is known that the redundancy of frames implies among other things the existence of many dual frames. While much attention has been paid to the so-called canonical dual frame, certain recent investigations have focused on alternate duals. For example, Sobolev duals were considered in \cite{BluLamPowYil2010, GunLamPowSaa2013} in relation to $\Sigma-\Delta$ quantization. Another example is the construction of dual frames for reconstruction of signals in the presence of erasures \cite{HanSun14}. In this paper, we introduce two other type of dualities, one dictated by the optimal transport problem, and the other grounded in the probabilistic setting we are working in. These two approaches will be developed in Section~\ref{sec2}.

The second goal of the paper is to investigate paths of probabilistic frames. Indeed, looking at the geodesic between any two probabilistic frames, it is natural to ask if the all probability measures along this path are probabilistic frames. This will be developed in Section~\ref{sec3}.

\section{Duality, Analysis, and Synthesis in the Set of Probabilistic Frames}\label{sec2}

\subsection{Transport Duals}
Duality, analysis, and synthesis are well-studied objects in finite frame theory. Sobolev duals have been proposed for use in reducing error in $\Sigma\Delta$ quantization \cite{BluLamPowYil2010}, and the authors of \cite{HanLengHuang} have found optimal dual frames for random erasures. Through the lens of optimal transport, extra nuances can be found in the probabilistic setting.

Given a frame $\Phi=\{\varphi_i\}_{i=1}^N$ as above, any possible dual frame to $\Phi$ can be written as: 
\begin{equation}\label{christensen_duals}
\{\psi_i\}_{i=1}^N = \{S_{\Phi}^{-1}\varphi_i + \beta_i - \sum\limits_{k=1}^N\ip{S_{\Phi}^{-1}\varphi_i}{\varphi_k}\beta_k\}_{i=1}^N
\end{equation}
where $\{\beta_i\}_{i=1}^N\subset\Rd$ and $S_{\Phi}$ is the frame operator for $\Phi$ \cite[Theorem 5.6.5]{Christensen2003}.  When $\beta_i=0$ for all $i,$ we have the canonical dual to $\Phi,$ which consists of the columns of the Moore-Penrose pseudoinverse of its analysis operator.  Inspired by the definition of duality above and this enumeration of the set of all possible duals to finite frames, we introduce a new notion of duality in the probabilistic context in this section.

\begin{definition}
Let $\mu$ be  a probabilistic frame on  $\Rd$. We say that a probability measure $\nu \in P_2(\Rd)$ is a \emph{transport dual} to $\mu$ if there exists $\gamma \in\Gamma(\mu,\nu)$ such that $$\iint_{\Rd \times \Rd}xy^{\top}d\gamma(x,y)=I.$$ 

We denote  the set of transport duals to $\mu$ by 
$$D_{\mu}:=\left\{\nu\in P_2(\Rd)\quad \Big|\quad \exists\space\gamma\in\Gamma(\mu,\nu)\textrm{ with } \iint_{\Rd \times \Rd}xy^{\top}d\gamma(x,y)=I\right\}.$$ 

\end{definition}
We let $\Gamma D_{\mu}\subset\Gamma(\mu,\nu)$ be  the set of joint distributions on $\RdxRd$ with first marginal $\mu$ ($\pi^1_{\hash}\gamma=\mu$) for which $\iint_{\Rd \times \Rd}xy^{\top}d\gamma(x,y)=I.$ This is the set of couplings (joint distributions) which induce the duality.

We recall that the canonical dual to a probabilistic frame $\mu$ defined in \cite{MartinRTF, KassoMartinOverview, MartinKassoPframe}, was given by $\tilde{\mu}:=(S_{\mu}^{-1})_{\#}\mu,$ yielding the reconstruction formula $ x=\int_{\Rd} \ip{x}{y}S_{\mu}y d\tilde{\mu}(y).$  It is easily seen that the canonical dual $\tilde{\mu}$ is an example of transport dual to $\mu$. Indeed, it is clear that $\gamma=(\iota\times S_{\mu}^{-1})_{\varhash}\mu \in \Gamma(\mu, \tilde{\mu}),$ where $\iota$ signifies the identity, and 
$$\iint_{\Rd \times \Rd}xy^{\top}d\gamma(x,y) = \int_{\Rd}x(S_{\mu}^{-1}x)^{\top}d\mu(x)=S_{\mu}S_{\mu}^{-1}=I.$$  Therefore, for a given probabilistic frame $\mu$, $\tilde{\mu} \in D_{\mu}.$

%This definition supersedes that given in previous works on probabilistic frames.  The canonical dual to a probabilistic frame $\mu$ defined in \cite{MartinRTF, KassoMartinOverview, MartinKassoPframe}, was given by $\tilde{\mu}:=(S_{\mu}^{-1})_{\#}\mu,$ yielding the reconstruction formula $ x=\int_{\Rd} \ip{x}{y}S_{\mu}y d\tilde{\mu}(y).$ This definition corresponds to the pushforward by the inverse of the frame operator and is contained within the new definition.  That formulation couples the probabilistic frame $\mu$ with what we term its canonical transport dual via the joint distribution: $$\gamma=(\iota\times S_{\mu}^{-1})_{\varhash}\mu.$$ 
%That is, $$\iint_{\RdxRd}xy^{\top}d\gamma(x,y) = \iint_{\RdxRd}x(S_{\mu}^{-1}x)^{\top}d\mu(x)=S_{\mu}S_{\mu}^{-1}=I.$$  
%This first example shows that our definition is not vacuous.
In fact, for a given probabilistic frame $\mu,$ there are other transport duals corresponding to similar deterministic couplings.  Generalizing the set of duals for discrete frames outlined in~\eqref{christensen_duals} leads to the following construction:

\begin{theorem}
Let $\mu$ be a probabilistic frame for $\Rd,$ and let $h:\Rd\rightarrow\Rd$ be any function in \\$\LL{\Rd}{\mu}:=\{f:\Rd\rightarrow\Rd | \int\norm{f(x)}_2^2d\mu(x) < \infty\}.$ Then $\psi_{h\hash}\mu\in D_{\mu},$ where $\psi_h:\Rd\rightarrow\Rd$ is defined by \\$\psi_h(x)=S_{\mu}^{-1}x+h(x)-\int_{\Rd}\ip{S_{\mu}^{-1}x}{y}h(y)d\mu(y).$
\end{theorem}

\begin{proof}
Consider $\mu,\psi_{h\hash}\mu$ as above.  Define $\gamma:=(\iota,\psi_{h})_{\hash}\mu\in\Gamma(\mu,\psi_{h\hash}\mu).$  

Then 
\begin{align*}
\iint_{\Rd \times \Rd}xy^{\top}d\gamma(x,y)	&=\int_{\Rd}x\left[S_{\mu}^{-1}x+h(x)-\int_{\Rd}\ip{S_{\mu}^{-1}x}{z}h(z)d\mu(z)\right]^{\top}d\mu(x)&\\
			&= I+\int_{\Rd} xh(x)^{\top}d\mu(x) -\iint_{\Rd\times \Rd}x(S_{\mu}^{-1}x)^{\top}zh(z)^{\top}d\mu(x)d\mu(z)&=I
\end{align*}
\end{proof}

The restriction of the set of transport duals $D_{\mu}$ to lie inside $P_2(\Rd)$ is necessary, unlike in the finite frame case.  One might consider the following simple example. Let $\{e_i\}_{i=1}^d\subset\Rd$ denote the standard orthonormal basis.  Let $\{\varphi_i\}_{i=1}^{d+1}$ be given by $\varphi_i=\sqrt{i2^i}e_i,$ $i\in\ixset{d},$ and let $\varphi_{d+1}=0.$  Take the weights $\alpha_i=\frac{1}{2^i}, i\in\N$. 
%, and let $\alpha_0=1-\sum\limits_{i=1}^d\frac{1}{2^i}.$  
Define $$\mu_1 = 2^{-d} \delta_0 + \sum\limits_{i=1}^d\alpha_i\delta_{\varphi_i}.$$  Let $\{\psi_i\}_{i=1}^{\infty}$ be given by $\psi_i=\sqrt{\frac{2^i}{i}}e_{1+[(i-1)\mod d]},\quad i\in\N.$  Let $\mu_2 = \sum\limits_{i=1}^\infty\alpha_i\delta_{\psi_i}.$
Then $\mu_1\in P_2(\Rd),$ but $$M_2^2(\mu_2)=\sum\limits_{i=1}^{\infty}\frac{1}{2^i}\norm{\psi_i}^2=\sum\limits_{i=1}^{\infty}\frac{1}{2^i}\frac{2^i}{i}=\infty.$$  Hence, $\mu_2\not\in P_2(\Rd).$  However, letting $\gamma\in P(\Rd \times \Rd)$ be given by $$\gamma = \sum\limits_{i=1}^d\alpha_i\delta_{(\varphi_i,\psi_i)}+\sum\limits_{i=d+1}^{\infty}\alpha_i\delta_{(0,\psi_i)},$$ it is clear that $\gamma\in\Gamma(\mu_1,\mu_2),$ and $$\iint_{\Rd \times \Rd}xy^{\top}d\gamma(x,y)=\sum\limits_{i=1}^d \frac{1}{2^i}\sqrt{i2^i}\sqrt{\frac{2^i}{i}}e_ie_i^{\top}=I.$$
This example shows that the Bessel-like restriction in the definition of transport duals, requiring them to lie in $P_2(\Rd),$ is necessary.  Given this restriction, we can assert the following theorem:

\begin{theorem}\label{duals_are_frames}
Let $\mu$ be a probabilistic frame.  Then:
\begin{enumerate}[label=\roman*]
\item[(i)] Each $\nu\in D_{\mu}$ is also a probabilistic frame.
\item[(ii)] $D_{\mu}$ is a compact subset of $P_2(\Rd)$ with respect to the weak topology. In particular, $D_{\mu}$ is a closed subset of $\PF$ with respect to the weak topology on $P_2(\Rd).$
\end{enumerate}
\end{theorem}

\begin{proof}

\begin{enumerate}[label=\roman*]
\item[(i)] Suppose $\nu\in D_{\mu}\subset P_2(\Rd).$  Since $D_{\mu}\subset P_2(\Rd)$ by definition, it is sufficient to show that $supp(\nu)$ spans $\Rd.$  

Let us assume, on the contrary, that there exists $z\in\Rd$, $z\neq 0$, such that $z\perp w$ for all $w\in\Span(\supp(\nu)).$ Pick $\gamma\in\Gamma(\mu,\nu)$ such that $\iint xy^{\top}d\gamma(x,y)=I.$   Because for all $x\in\supp(\nu),$ $z^{\top}x=0,$ 
$$ \norm{z}^2	 = 	\iint \ip{z}{x}\ip{z}{y}d\gamma(x,y) = 	\iint \ip{z}{x}\ip{z}{y}\indicator{\supp(\nu)\times\Rd}(x,y)d\gamma(x,y) =  0 $$ which is a contradiction.

\item[(ii)] Consider the lifting of the dual set, $\Gamma D_{\mu }:= \{\gamma\in\Gamma(\mu,\nu)\textrm{ s.t. }\iint xy^{\top}d\gamma(x,y)=I\}.$  It can be shown by Prokhorov's Theorem that $\Gamma D_{\mu}$ is precompact \cite[Chapter 4]{Villani2009}. That is, given $\{\gamma_n\}\subset \Gamma D_{\mu},$ there exists a subsequence $\{\gamma_{n_k}\}$ converging weakly to a limit $\gamma\in P_2(\Rd \times \Rd).$  With this in mind, if $\{\nu_n\}$ is a sequence in $D_{\mu},$ we can choose the corresponding $\{\gamma_n\}\in\Gamma D_{\mu},$ and let $\{\nu_{n_k}\}$ be the second marginals of a subsequence $\{\gamma_{n_k}\}.$ For all $\varphi\in C(\RdxRd)$ satisfying for some $C>0$ $|\varphi(x,y)|\leq C(1+\norm{x}^2+\norm{y}^2),$ $$\iint\varphi(x,y)d\gamma_{n_k}(x,y)\longrightarrow\iint\varphi(x,y)d\gamma(x,y).$$  

In particular, for all such $\varphi=\varphi(x),$ $$\iint \varphi(x)d\gamma_{n_k}(x,y)=\int\varphi(x)d\nu_{n_k}(x)\longrightarrow\iint\varphi(x)d\gamma(x,y)=\int\varphi(x)d(\pi^2_{\hash}\gamma)(x).$$

Thus $\nu_{n_k}$ converges weakly in $P_2(\Rd)$ to $\pi^2_{\hash}\gamma =: \nu,$ so that $\{\nu_n\}$ contains a weakly convergent subsequence.  Therefore $D_{\mu}$ is precompact.

Now let $\{\nu_n\}$ be any convergent sequence in $D_{\mu}$ which has a limit $\nu$ and which forms the second marginals of $\{\gamma_n\}\subset \Gamma D_{\mu}.$  Take again a convergent subsequence $\{\gamma_{n_k}\}$ with limit $\gamma$ necessarily in $\Gamma(\mu,\nu).$  Since $|x_i y_j|\leq\half(\norm{x}^2+\norm{y}^2),$ it follows that $$\iint x_i y_j d\gamma_{n_k}(x,y)\longrightarrow\iint x_i y_j d\gamma(x,y).$$  Then, since for each $n_k,$ $\iint x_i y_j d\gamma_{n_k}(x,y)\equiv \delta_{i,j},$ it follows that $\iint x_i y_j d\gamma(x,y)= \delta_{i,j},$ and therefore $\nu\in D_{\mu}.$  This shows that $D_{\mu}$ is also closed, and is therefore compact.  The closedness in $PF(\Rd)$ then follows naturally.

\end{enumerate}
\end{proof}

From the definition of transport duals, it is clear that their construction depends on the creation of a probability distribution on the product space which has a predetermined second-moments matrix and first and second marginals. This is, in general, a very difficult problem, which becomes a bit more tractable for probabilistic frames supported on finite, discrete sets by appealing to tools from linear algebra. 

Suppose we have two frames $\Phi=\{\varphi_i\}_{i=1}^N$ and $\Psi=\{\psi_j\}_{j=1}^M$, and two sets of positive weights, $\{\alpha_i\}_{i=1}^N$ and $\{\beta_j\}_{j=1}^M$, summing to unity.  Let $\mu_{\Phi,\alpha}:=\sum_{i=1}^N\alpha_i\delta_{\varphi_i}$, and let $\mu_{\Psi,\beta}:=\sum_{j=1}^M\beta_j\delta_{\psi_j}$.
In this case, any joint distribution $\gamma$ for $\mu_{\Phi,\alpha}$ and $\mu_{\Psi,\beta}$ satisfies 
$$d\gamma(x,y)=\sum\limits_{i=1}^N\sum\limits_{j=1}^M A_{i,j}\delta_{\varphi_i}(x)\delta_{\psi_j}(y)$$
where $A\in\R^{N\times M}$ with $$\sum_{i=1}^N A_{i,j}=\beta_j,\quad \sum_{j=1}^M A_{i,j}=\alpha_i,\quad A_{i,j}\geq 0\quad \forall i,j, \textrm{ and } \sum_{i=1}^N\sum_{j=1}^N A_{i,j}=1.$$ That is, there is a one-to-one correspondence between $\Gamma(\mu_{\Phi,\alpha},\mu_{\Psi,\beta})$ and this set of ``doubly stochastic'' matrices, which we denote by $DS(\alpha,\beta).$  Thus, to show that $\mu_{\Phi,\alpha}\in D_{\mu_{\Psi,\beta}}$, one must construct a matrix $A\in DS(\alpha,\beta)$ solving $\Phi^{\top}A\Psi=I.$

Regarding this question, we have the following result:

\begin{theorem}
Given frames $\{\varphi_i\}_{i=1}^N$ and $\{\psi_j\}_{j=1}^M$ for $\Rd$ with analysis operators $\Phi$ and $\Psi$, there exists $A\in DS(\alpha,\beta)$ with $\Phi^{\top}A\Psi = I$ if and only if there is no  triplet $(B, u, v)$ with  $B\in\R^{d\times d},$ $u\in\R^M,$ $v\in\R^N$ such that
\begin{empheq}[left=\empheqlbrace]{align*}
\quad\varphi_i^{\top}B\psi_j + u_i + v_j &\geq 0\\
\quad trace(B) + u^{\top}\alpha +v^{\top}\beta & < 0
\end{empheq}
\end{theorem}

\begin{proof}
Recall that we must solve the system
\begin{align}
\Phi^{\top}A\Psi = I, \quad A_{i,j} \geq 0\ \quad \sum_{j=1}^M A_{i,j} = \alpha_i, \quad \sum_{i=1}^N A_{i,j} = \beta_j \label{A_problem_statement}
\end{align}

Defining, for a matrix $B$, $vec(B)$ to be the vector formed by stacking the columns of $B$, we may rewrite the problem in terms of the Kronecker product.  Using the following variables, $K=\Psi^{\top}\otimes\Phi^{\top},$ $a=vec(A),$ $z_N=[1 \dots 1]^{\top}\in\R^N$, $z_M=[1 \dots 1]^{\top}\in\R^M$, and $t=vec(I),I\in\R^{d\times d}$, we have:
%\begin{subequations}
%\begin{align*}
\begin{empheq}[left=\empheqlbrace]{align*}
\quad Ka &= t \\
\quad (z_N^{\top} \otimes I_{(M\times M)})a &=\beta \\
\quad (I_{(N\times N)} \otimes z_M^{\top})a &=\alpha \\
\quad a_i & \geq  0 \quad \forall i\in\{1, \dots, MN\}
\end{empheq}
%\end{align*}
%\end{subequations}

\noindent We can combine the equations above, letting $$K'=\left[\begin{array}{c} K \\ (z_N^{\top} \otimes I_{(M\times M)}) \\ (I_{(N\times N)} \otimes z_M^{\top})\end{array}\right] \textrm{ and } t'=\left[\begin{array}{c} t \\ \beta \\ \alpha \end{array}\right]$$  Then the problem simplifies to solving $K'a = t'$ such that $a_i \geq  0 \quad \forall i\in\{1, \dots, MN\}.$
%\begin{subequations}
%\begin{empheq}[left=\empheqlbrace]{align*}
%\quad K'a &= t'\\
%\quad a_i & \geq  0 \quad \forall i\in\{1, \dots, MN\}
%\end{empheq}
%\end{subequations}
By Farkas' Lemma \cite[Lemma 1]{FarkasRef}, either this system has a solution or there exists $y\in\R^{d^2+M+N}$ such that 
%\begin{subequations}
%\begin{align}
\begin{empheq}[left=\empheqlbrace]{align}
\quad y^{\top}K' &\geq 0 \label{farkas_statement1}\\
\quad y^{\top}t' &< 0 \label{farkas_statement2}
\end{empheq}
%\end{align}
%\end{subequations}
Now write any such $y$ as $y=\left[\begin{array}{c} b \\ u \\ v \end{array} \right],$ with $b\in\R^{d^2},$ $u\in\R^M,$ and $v\in\R^N$, and let $b=vec(B)$ with $B\in\R^{d\times d}.$  Then Equations \eqref{farkas_statement1} and \eqref{farkas_statement2} hold if and only if
\begin{subequations}
\begin{align*}
b^{\top}K + u^{\top}(z_N^{\top} \otimes I_{(M\times M)}) +v^{\top}(I_{(N\times N)} \otimes z_M^{\top}) &\geq 0\\
b^{\top}t + u^{\top}\beta + v^{\top}\alpha &< 0
\end{align*}
\end{subequations} 
That is,
\begin{subequations}
\begin{align*}
vec(\Phi B\Psi^{\top}) + vec(z_N u^{\top} I_{(M\times M)})^{\top} +vec(I_{(N\times N)} v z_M^{\top})^{\top} &\geq 0\\
b^{\top}t + u^{\top}\beta + v^{\top}\alpha &< 0
\end{align*}
\end{subequations} 
or, equivalently,
\begin{subequations}
\begin{align*}
\varphi_i^{\top} B\psi_j + u_i + v_j  &\geq 0 \quad {\forall i,j}\\
trace(B) + u^{\top}\beta + v^{\top}\alpha &< 0
\end{align*}
\end{subequations} 
\end{proof}

The simplicity of the following corollary is alluring because it connects the weighted averages of the frame vectors to the existence of the plan yielding the duality, but the condition is difficult to show because of its scope.  
\begin{corollary}
If there does not exist $B\in\R^{d\times d}$ such that
\begin{align*}\label{farkas_equiv}
\alpha^{\top}\Phi B\Psi^{\top}\beta - trace(B) &> 0
\end{align*}
then by Farkas' Lemma the system of \eqref{A_problem_statement} (and its equivalents) is not solvable and the desired matrix $A\in DS(\alpha,\beta)$ exists.
\end{corollary}

\noindent A true converse has proven elusive. However, we can identify a few related conditions under which no transport duals whatsoever can be constructed.  In particular, in the case that the frames are uniformly weighted, we have the following zero-centroid condition.

\begin{theorem}\label{zero_centroid}
Again, take $z_N:=[1\dots 1]^{\top}\in\R^N.$ Suppose that  $\Psi=\{\psi_i\}_{i=1}^N\subset\Rd$ is a frame such that $\sum_{i=1}^N\psi_i=0$, then $\mu_{\Psi,\frac{1}{N}z_N}$ has no equal-weight transport dual supported on a set of of cardinality $d.$
\end{theorem}

\begin{proof}
Given $\Psi$ as above, let $\{v_j\}_{j=1}^d\subset\mathbf{R}^N$ denote the columns of the analysis operator $\Psi,$ and let $\{u_i\}_{i=1}^d\subset\mathbf{R}^N$ denote the rows of some $A\in DS(\frac{1}{d}z_d,\frac{1}{N}z_N).$ $\Psi$ will have a transport dual of cardinality $d$ if and only if for some $A,$ $A\Psi = [[\ip{u_i}{v_j}]]$ is invertible. (Here, $Q=[[q_{i,j}]]$ denotes the entrywise definition of $Q.$) Each $u_i=s +\lambda^i,$ where $s=[\frac{1}{Nd} \cdots \frac{1}{Nd}]\in\Rd,$ and
\begin{align}
\sum\limits_{k=1}^N \lambda_k^i & = 0 \textrm{ for each } i\in\{1,...,d\}\nonumber \\ 
\sum\limits_{i=1}^d\lambda_k^i  & = 0 \textrm{ for each } k\in\{1,...,N\} \nonumber
\end{align}
\noindent so that $\{\lambda^i\}_{i=1}^d$ has zero centroid as well and is therefore linearly dependent.  Let $\Lambda = [\lambda^1 \hdots \lambda^d]^{\top}.$  Then $$det(A\Phi)  = \prod_{i=1}^d \ip{u_i}{v_i}	= \prod_{i=1}^d \ip{s+\lambda^i}{v_i} = \prod_{i=1}^d \ip{\lambda^i}{v_i}= det(\Lambda\Psi) = 0$$			
\noindent because $v_i \perp s$ for all $i\in\{1,\dots,d\}$ and since $rank(\Lambda) \leq d-1.$
\end{proof}

As a consequence, Theorem \ref{zero_centroid} implies that no equiangular tight frame in $\mathbf{R}^2$ has a transport dual of cardinality 2.

\begin{remark}
One interesting aspect of the transport duals in the context of finite discrete probabilistic frames, i.e., finite frames, is the existence of pairs of dual frames with different cardinalities. For example, one can consider the probabilistic frame given by $d\mu=\frac{1}{2}\delta_{\varphi_1}+\frac{1}{6}\delta_{\varphi_2}+\frac{1}{3}\delta_{\varphi_3}$ with $\varphi_1=[1\quad 0]^{\top},$ $\varphi_2=[\frac{\sqrt{3}}{2} \quad \frac{1}{2}]^{\top},$ and $\varphi_3=[0 \quad 1]^{\top}.$ Then the probabilistic frame $\nu$ given by $d\nu=\frac{1}{2}\delta_{\psi_1}+\frac{1}{2}\delta_{\psi_2}$ with $\psi_1=[\frac{18}{4-\sqrt{3}}\quad \frac{6\cdot(2+\sqrt{3}}{4-\sqrt{3}}]^{\top}$ and $\psi_2=[\frac{-12}{4-\sqrt{3}}\quad \frac{24}{4-\sqrt{3}}]^{\top}$ is a transport dual for $\mu$ with support of different cardinality. The role of  transport duals in problems such as reconstruction in the presence of erasure will be the object of future investigations. 
\end{remark}

\subsection{Analysis and Synthesis in the Probabilistic Context}

In \cite{MartinRTF, KassoMartinOverview, MartinKassoPframe}, the analysis and synthesis operators for probabilistic frames were defined analogously to those of continuous frames.  Given a probabilistic frame $\mu$, the analysis operator was defined \cite[2.2]{KassoMartinOverview} as  $$A_{\mu}:\Rd\rightarrow\LL{\Rd}{\mu} \textrm{ given by  } x\mapsto\ip{x}{\cdot}.$$  Its synthesis operator was $$A_{\mu}^{*}:\LL{\Rd}{\mu}\rightarrow\Rd \textrm{ given by } f\mapsto\int_{\Rd}xf(x)d\mu(x).$$ 
The foregoing construction of transport duals, on the other hand, begs a more probability-theoretic definition of analysis and synthesis.  As defined above, the analysis operator $A_{\mu}$ is independent of the measure $\mu$.  Indeed, it is not clear from this definition how one could do ``analysis'' with one probabilistic frame followed by ``synthesis'' with another.  However, finite frame theory itself gives us a clue about how to think about analysis and synthesis in the probabilistic context.  
\begin{example}
Consider two frames for $\Rd$, $\{\varphi_i\}_{i=1}^N$ and $\{\psi_i\}_{i=1}^N$. Let $\{e_i\}_{i=1}^N\subset\R^N$ be an orthonormal basis for $\R^N.$ Then the analysis operator for $\Phi$, $A_{\Phi}:\Rd\rightarrow\R^N$ is given by $$A_{\Phi}(x)=\Phi x =\sum\limits_{i=1}^N\ip{x}{\varphi_i}e_i\quad \textrm{for } x\in\Rd.$$  The synthesis operator for $\Psi$, $A_{\Psi}^*:\R^N\rightarrow\Rd$, is given by $$A_{\Psi}^*(y)=\Psi^{\top}y=\sum\limits_{i=1}^N \ip{y}{e_i}\psi_i\quad\textrm{for }y\in\R^N.$$  Then we can compose the operators simply by writing $A_{\Psi}^*A_{\Phi}(x)=\sum\limits_{i=1}^N\ip{x}{\varphi_i}\psi_i.$
If, however, we choose some $\sigma$ and $\pi$ in $\Pi_N,$ the set of permutations on N-element sets, and instead choose to do analysis and synthesis with the two frames as $$A_{\Psi}^*A_{\Phi}(x)=\sum\limits_{i=1}^N\ip{x}{\varphi_{\sigma(i)}}\psi_{\pi(i)},$$ then it will be as if we had chosen two different finite frames to work with.  This is because the ordering of the frame vectors is implicitly tied to the ordering of the reference basis $\{e_i\}_{i=1}^N.$  
\end{example}

Order matters! From the example, it is clear that even given the fixed reference basis, we cannot truly speak of a single analysis operator for the set $\{\varphi_i\}_{i=1}^N,$ without imposing an order on it relating it to the fixed reference basis.  Similarly, for a probabilistic frame $\mu,$ there must be a reference \textit{measure} $\eta$ playing the role of the reference basis, and this will still lead to a family of analysis operators, each corresponding to a joint distribution $\gamma\in\Gamma(\mu,\eta).$  The orthogonality of the reference basis in the above example turns out not to be necessary; its function is to match up frame coefficients with the appropriate vectors.  What is key is that transport plans exist between the probabilistic frame and the reference measure and that the support of the reference measure is sufficient to ``glue'' together arbitrary probabilistic frames through analysis and synthesis.

To make this idea of coefficient-matching rigorous, some technicalities about conditional probabilities are necessary.  Conditional probabilities can be defined via the Rokhlin Disintegration Theorem \cite[Theorem 5.3.1]{AGS2005}.  If $\boldsymbol{\mu}\in P(\R^M\times\R^N)$ and $\nu=\mu^1=\pi^1_{\hash}\boldsymbol{\mu}$, then one can find a Borel family of probability measures $\{\mu_x\}_{x\in\R^M}\subset P(\R^N)$ which is $\mu^1$-a.e. uniquely determined such that $\boldsymbol{\mu}=\int_{\R^M}\mu_{x}d\mu^1(x)$.  In the language of conditional probability, for any $f\in C_b(\R^M\times\R^N)$, it is then meaningful to write $$\iint_{\R^M\times\R^N}f(x,y)d\boldsymbol{\mu}(x,y)=\int_{\R^N}\int_{\R^M}f(x,y)d\boldsymbol{\mu}(y|x)d\mu^1(x),$$ with the understanding that $\boldsymbol\mu(\cdot | x)$ is defined $\mu^1$-a.e.  Gluings can then be constructed, which allow us to use conditional probabilities with respect to a common reference measure to construct a joint distribution between previously unrelated measures. 
\begin{lemma}\emph{Gluing Lemma \cite[Lemma 5.3.2]{AGS2005}}\label{gluings}
Let $\gamma^{12}\in P(\R^K\times\R^M)$, $\gamma^{13}\in P(\R^K\times\R^N)$ such that $\pi^1_{\hash}\gamma^{12}=\pi^1_{\hash}\gamma^{13}=\mu^1$.  Then there exists $\boldsymbol{\mu}\in P(\R^K\times\R^M\times\R^N)$ such that $\pi^{1,2}_{\hash}\boldsymbol{\mu}=\gamma^{12}$ and $\pi^{1,3}_{\hash}\boldsymbol{\mu}=\gamma^{13}$.  Moreover, if $\gamma^{12}=\int_{\R^K} \gamma_{x_1}^{12}d\mu^1$, $\gamma^{13}=\int_{\R^K} \gamma^{13}_{x_1}d\mu^1$, and $\boldsymbol{\mu}=\int_{\R^K} \boldsymbol{\mu}_{x_1}d\mu^1$ are the disintegrations of $\gamma^{12}$, $\gamma^{13}$, and $\boldsymbol{\mu}$ with respect to $\mu^1$, then the first statement is equivalent to $\boldsymbol{\mu}_{x_1}\in \Gamma(\gamma_{x_1}^{12},\gamma_{x_1}^{13})\subset P(\R^M\times\R^N)$ for $\mu^1$-a.e. $x_1\in\R^K$. 
\end{lemma}

Now let us consider a probabilistic frame $\mu$ and another probability measure $\eta$ and take $\gamma \in\Gamma(\mu,\eta)$.  From Lemma \ref{gluings}, there is a set of conditional probability measures $\{\gamma(\cdot|w)\}_{w\in\Rd}$ that are uniquely defined $\eta$-a.e. To proceed with the construction of analysis and synthesis in the probabilistic context, we will first establish a useful fact. Recall that  $$\LL{\RN\times\RM,\RK}{\gamma}:=\{f:\RN\times\RM\rightarrow\RK\quad |\quad \iint \norm{f(x,y)}_2^2 d\gamma(x,y)<\infty\}.$$ Then, by condition Jensen's inequality, if f $f\in \LL{\RdxRd,\Rd}{\gamma},$ it follows that $g(w):=\int_{\Rd}f(y,w)d\gamma(y|w)$ is in $\LL{\Rd,\Rd}{\eta}$.

%\begin{proposition}
%If $f\in \LL{\RdxRd,\Rd}{\gamma},$ it follows that $g(w):=\int_{\Rd}f(y,w)d\gamma(y|w)$ is in $\LL{\Rd,\Rd}{\eta}$.
%\end{proposition}
%
%\begin{proof}
%By conditional Jensen's inequality,
%\begin{align*}
%\int\norm[\bigg]{\int f(y,w)d\gamma(y|w)}_2^2d\eta(w) &\leq \iint \norm[\Big]{f(y,w)}_2^2 d\gamma(y|w)d\eta(w)\\
%										&=\iint \norm[\Big]{f(y,w)}_2^2 d\gamma(y,w)\\
%										&=\norm{f}^2_{\LL{\RdxRd,\Rd}{\gamma}}\\
%										&<\infty.
%\end{align*}
%\end{proof}

\noindent Finally, since $h(z,w):=\norm{z}_2\in\LL{\RdxRd,\R}{\gamma}$ for any $\gamma\in\Gamma(\mu,\eta)$ provided that $\mu\in P_2(\Rd)$, it follows that the vector-valued function $\int zd\gamma(z|w)$ lies in $\LL{\Rd,\Rd}{\eta}$.

To define analysis and synthesis operators which are more closely tied to their probabilistic frames, a reference measure must be chosen; take an absolutely continuous $\eta \in P_2(\Rd)$ whose support is $\Rd.$ Given $\mu\in \PF,$ we define families of analysis and synthesis operators for $\mu$ with respect to $\eta$.

\begin{definition}\label{analysis_synthesis_defn}
$\{A_{\mu}^{\gamma}\}_{\gamma\in\Gamma(\mu,\eta)}$  is the family of analysis operators, and for each $\gamma \in \Gamma(\mu,\eta)$ we have: $A_{\mu}^{\gamma}:\Rd\rightarrow \LL{\Rd,\Rd}{\eta}$, is given by
$$A_{\mu}^{\gamma}(x)(w)=\int_{\Rd}\ip{x}{y}d\gamma(y|w).$$
Similarly, the family of synthesis operators, $\{Z_{\mu}^{\gamma}\}_{\gamma\in\Gamma(\mu,\eta)}$ is defined for each $\gamma \in \Gamma(\mu,\eta)$ by $Z_{\mu}^{\gamma} :\LL{\Rd,\R}{\eta}\rightarrow\Rd$, given by
$$Z_{\mu}^{\gamma}(f)	=\iint_{\RdxRd} zf(w)d\gamma(z|w)d\eta(w)$$ 
\end{definition}

\noindent The class of reference measure $\eta$ was chosen such that, for any probabilistic frame $\mu,$ the probabilistic analysis and synthesis operators can be constructed using deterministic couplings between $\eta$ and $\mu$.
%, since these are often the most tractable sort of coupling to work with.  

There are several interesting ways to pair disparate types of probabilistic frames with one another. A useful technique is the transport of an absolutely continuous measure to a discrete measure using power (Voronoi) cells. Following \cite{Merigot}, we define maps which can be used for these pairings.  It is an interesting fact due to Brenier that the Voronoi mapping we will describe, $T_P^w,$ is in fact an optimal map between the two measures it couples, $\mu$ and $T_P^w|_{\hash}\mu,$ when $\mu$ is absolutely continuous with respect to Lebesgue measure \cite[Theorem 1]{Merigot}.

\begin{definition}\label{voronoi_diagram}
Given a probability measure $\mu$ on $\Rd$, a finite set $P$ of points in $\Rd$ and $w:P\rightarrow\R_{+}$ a weight vector, the power diagram or weighted Voronoi diagram of $(P,w)$ is a decomposition of $\Rd$ into cells corresponding to each member of $P$.  Given $p\in P$, a point $x \in \Rd$ belongs to $\Vor_P^w(p)$ if and only if for every $q\in P$, $$\norm{x-p}^2-w(p)\leq \norm{x-q}^2-w(q).$$

Let  $T_P^w$ be the map that assigns to  each  $x$ in a power cell $\Vor_P^w(p)$ to $p$, the ``center'' of that power cell. We call $T_P^w$  the weighted Voronoi mapping.  $$T_P^w|_{\hash}\mu=\sum\limits_{p\in P}\mu(\Vor_P^w(p))\delta_p.$$

Let $\eta$ be an absolutely continuous measure in $P_2(\Rd)$, and let $\nu=\sum_{p\in P}\lambda_p\delta_p$ be a discrete measure in $P_2(\Rd)$ supported on a finite set of points $P$ with weights $\{\lambda_p\}$ summing to unity.  Then we say that a vector weight $w:P\rightarrow\R_{+}$ is adapted to $(\eta,\nu)$ if for all $p\in P$, $\lambda_p=\eta(\Vor_P^w(p))=\int_{\Vor_P^w(p)}d\eta(x)$.
\end{definition}

\begin{example}
Now given discrete frames $\Phi=\{\varphi_i\}_{i=1}^M$ and $\Psi=\{\psi_j\}_{j=1}^N$ for $\Rd$, and $\eta$ a reference measure in Definition \ref{analysis_synthesis_defn}, choose $\gamma_1=(\iota,T_{\Phi}^{w_1})_{\hash}\eta$ and $\gamma_2=(\iota,T_{\Psi}^{w_2})_{\hash}\eta$, where the weights $w_1$ and $w_2$ are adapted to $(\mu_{\Phi},\eta)$ and $(\mu_{\Psi},\eta)$, respectively.  Then $$Z_{\mu_{\Psi}}^{\gamma_2}(A_{\mu_{\Phi}}^{\gamma_1}(x))=\int\ip{x}{T_{\Phi}^{w_1}(y)}T_{\Psi}^{w_2}(y)d\eta(y).$$
\end{example}

\begin{example}\emph{Recovering the old definitions of analysis and synthesis}

In the special case $M=N$, we could choose $P=\{p_i\}_{i=1}^N\subset\Rd$ and $w_0$ adapted to $(\mu_P,\eta)$. Then let $f_{\Psi}:P\rightarrow\Psi$  be given by $f_{\Psi}(p_i)=\psi_i$, and let $f_{\Phi}:P\rightarrow\Phi$ be similarly defined. Then if $\gamma_1=(\iota,f_{\Phi}\circ T_{P}^{w_0})_{\hash}\eta$ and $\gamma_2=(\iota,f_{\Psi}\circ T_{P}^{w_0})_{\hash}\eta$, it follows that $$Z_{\mu_{\Psi}}^{\gamma_2}(A_{\mu_{\Phi}}^{\gamma_1}(x))=\int\ip{x}{f_{\Phi}\circ T_P^{w_0}(y)}f_{\Psi}\circ T_P^{w_0}(y)d\eta(y)=\sum_{i=1}^N\ip{x}{\varphi_i}\psi_i.$$
Hence, we have recovered the analysis and synthesis operation of finite frames.
\end{example}

\begin{example}\emph{Discrete dual to absolutely continuous probabilistic frame}

Finally, choose a frame contained in the support of $\eta$, say $\{\psi_i\}_{i=1}^N$.  Let $T_{\Psi}^w$ be the transport map between $\eta$ and $\mu_{\Psi}$, as constructed above. Choose $\{\varphi_i\}_{i=1}^N$ to be any dual to $\{\psi_i\}_{i=1}^N$, and let $f:\Psi\rightarrow\Phi$  be given by $f(\psi_i)=\varphi_i$. Then $\gamma=(\iota,f\circ T_{\Psi}^w)_{\hash}\eta \in P_2(\RdxRd)$ is a joint transport plan in $\Gamma(\eta,\mu_{\Psi})$ such that $\iint xy^{\top}d\gamma(x,y)=\int x T_{\Psi}^w(x)d\eta(x)=I$, so that $\eta$ and $\mu_{\Psi}$ are dual to one another in $\PF$.
\end{example}

\section{Paths of Frames: Geodesics for the Wasserstein Space}\label{sec3}
A number of important questions in finite frame theory involve determining distances between frames and constructing new frames. In this section we consider geodesics in $P_2(\Rd)$ and investigate conditions under which probability measures on these paths are probabilistic frames. As we shall prove, in  the case of discrete probabilistic frames, this question is equivalent to one of ranks of convex combinations of matrices. Furthermore, for probabilistic frames with density, a sufficient condition for geodesic measures to be probabilistic frames is the continuity of the optimal deterministic coupling. This question has ramifications for constructions of paths of frames in general, for frame optimization problems, and for our understanding of the geometry of $PF(\Rd).$

%Since the space $P_2(\Rd)$ can have some nonintuitive features, the first step is to consider geodesics in the Wasserstein space $P_2(\Rd)$ and to establish conditions under which every measure on the geodesic path between two probabilistic frames is itself a probabilistic frame.
\subsection{Wasserstein Geodesics}
In constructing paths of probabilistic frames, minimal paths between frames in $P_2(\Rd)$ are a natural place to start since $\PF$ is not closed. We follow the construction of geodesics in the Wasserstein space given in \cite{Gangbo2004}. To this end, given $t\in [0, 1]$ define  $\Pi^t:\RdxRd\rightarrow\Rd$  as $\Pi^t(x,y) = (x, (1-t)x+ty).$ For  $\mu_0, \mu_1 \in P_2(\Rd)$, take  $\gamma_0 \in \Gamma(\mu_0,\mu_1)$ to be  an optimal transport plan for $\mu_0$ and $\mu_1$ with respect to the 2-Wasserstein distance. Then let the interpolating joint probability measure be $\gamma^t$ on $\Rd \times \Rd,$ given by:
$$\iint_{\RdxRd}F(x,y)d\gamma^t(x,y) = \iint_{\RdxRd}F(\Pi^t(x,y))d\gamma_0(x,y)$$

\noindent for all $F \in C_b(\RdxRd).$ In particular, for $F \in C_b(\Rd)$,  $$\iint_{\RdxRd}F(x)d\gamma^t(x,y) = \iint_{\RdxRd}F(x)d\gamma_0(x,y) = \int_{\Rd}F(x)d\mu_0(x).$$ 

\noindent Given $t\in [0,1]$  let  $\mu_t $  be the probability measure such that for all $G\in C_b(\Rd)$: 
\begin{equation}\label{geodesic_meas}
\int_{\Rd}G(y)d\mu_t(y)=\iint_{\RdxRd}G(y)d\gamma^t(x,y) = \iint_{\RdxRd}G((1-t)x+ty)d\gamma_0(x,y),
\end{equation}
\noindent we call $\mu_t$ a geodesic measure with respect to $\mu_0$ and $\mu_1$.  Indeed, the mapping $t\rightarrow \mu_t$ is truly a geodesic of the 2-Wasserstein distance in the sense that $$W_2(\mu_0, \mu_t)+W_2(\mu_t,\mu_1)=W_2(\mu_0,\mu_1).$$ 

Recall that a probability measure $\mu$ on $\Rd$ is a probabilistic frame if it is an element of $P_2(\Rd)$ and if $S_{\mu}$ is positive definite. It is easy to show that $\mu_t$, as constructed  by the method above, always meets the first requirement.

\begin{lemma}\label{path_bound}
For any measure $\mu_t$, $t\in [0,1]$, on the geodesic between two probabilistic frames $\mu_0$ and $\mu_1,$ $M_2^2(\mu_t)<\infty.$ 
\end{lemma}

Showing that $S_{\mu_t}$ is positive definite, or, equivalently, that the support of $\mu_t$ spans $\Rd$ depends on the characteristics of the support of the measures at the endpoints.  For this reason, it is natural to divide the analysis into two parts: the discrete case and the absolutely continuous case. In both, a monotonicity property that characterizes optimal transport plans will play a key role. 
%We do not touch on singular continuous measures in this treatment.

\subsection{Probabilistic Frames with Discrete Support}

For the canonical discrete probabilistic frames with uniform weights, we have:
\begin{lemma}\emph{\cite[Theorem 6.0.1]{AGS2005}}
Given $\mu_0=\mu_{\Phi}$ and $\mu_1=\mu_{\Psi},$ discrete probabilistic frames with supports of equal cardinality $N,$ uniformly weighted, the Monge-Kantorovich problem simplifies, and denoting by $\Gamma(\frac{1}{N})$ the set of matrices with row and column sums identically $\frac{1}{N}$:
$$W_2^2(\mu_0,\mu_1) = \min_{A\in \Gamma(\frac{1}{N})}\sum\limits_{i=1}^N\sum\limits_{j=1}^N a_{i,j}\norm{\varphi_i-\psi_j}^2$$
and, by the Birkhoff-von Neumann Theorem, the optimal transport matrix $A$ is a permutation matrix corresponding to some $\sigma\in \Pi_N$, i.e.:
$$W_2^2(\mu_0,\mu_1) =\min_{\sigma\in \Pi_N} \frac{1}{N} \sum\limits_{i=1}^N \norm{\varphi_i-\psi_{\sigma(i)}}^2$$
\end{lemma}
\noindent In this case, for some optimal $\sigma\in \Pi_N$, 
\begin{equation}\label{DPF_sum_form}
S_{\mu_t}:=\frac{1}{N}\Sum{i}{N}[(1-t)\varphi_i+t\psi_{\sigma(i)}][(1-t)\varphi_i+t\psi_{\sigma(i)}]^{\top}.
\end{equation}

\noindent The optimality of $\sigma$ implies that $\sigma$ maximizes $\sum\limits_{i=1}^N\ip{\varphi_i}{\psi_{\sigma(i)}}$ among all elements of $\Pi_N,$ and this crucial fact motivates the definition of a monotonicity condition. 

\begin{definition}
A set $S\subset \RdxRd$  is said to be cyclically monotone if, given any finite subset $\{(x_1,y_1),...,(x_N,y_N)\}\subset S,$ for every $\sigma\in S_N$ holds the inequality: $$\sum_{i=1}^N \ip{x_i}{y_i} \geq \sum_{i=1}^N \ip{x_i}{y_{\sigma(i)}}.$$
\end{definition}

With this definition in hand, the main result of this section can be stated:

\begin{theorem}\label{first_geod_thm}
Let $\{\varphi_i\}_{i=1}^N$ and $\{\psi_i\}_{i=1}^N$ be frames for $\Rd$. If $\Psi^{\dagger}\Phi$ has no negative eigenvalues and $\{(\varphi_i,\psi_i)\}_{i=1}^N$ is cyclically monotone, then every measure on the geodesic between the canonical probabilistic frames $\mu_{\Phi}$ and $\mu_{\Psi}$ is a probabilistic frame.
\end{theorem}

The proof of this theorem will follow from  Lemma \ref{path_bound} and Proposition \ref{proposition_DP}, proven below. To prove Proposition \ref{proposition_DP}, the following lemma from matrix theory is necessary:

\begin{lemma}{\cite[Theorem 2]{Szulc1996}}\label{szulc_lemma}
Let $A$ and $B$ be $m\times n$ complex matrices, $m\geq n$. Let $rank(A)=rank(B)=n$. If $B^{\dagger}A$ has no nonnegative eigenvalues, then every matrix in $$h(A,B):=\{(1-t)A+tB,\quad t\in[0,1]\}$$ has rank $n$.  Similarly, if $A$ and $B$ are $n\times n$ complex matrices with rank $n$, we can define in $$r(A,B):=\{(I-T)A+TB\},$$ where $T$ is a real diagonal matrix with diagonal entries in $[0,1]$.  Then, if $B^{-1}A$ is  such that all its principal minors are positive, then every matrix in $r(A,B)$ will have rank $n$.
\end{lemma}

Combining the cyclical monotonicity condition with Lemma \ref{szulc_lemma}, we can state the following result which gives sufficient conditions for a geodesic between discrete probability measures in $P_2(\Rd)$ to be a path of frames. We note that little can be claimed about the spectra of the frame operators along the path (i.e., the frame bounds of the probabilistic frames along the geodesic) in general, other than their boundedness away from zero.

\begin{proposition}\label{proposition_DP}
Let $\{\varphi_i\}_{i=1}^N$ and $\{\psi_i\}_{i=1}^N$ be frames for $\Rd$ with analysis operators $\Phi$ and $\Psi$. Denoting by $Psi^{\dagger}$ the Moore-Penrose pseudoinverse of $\Psi,$ if $\Psi^{\dagger}\Phi$ has no negative eigenvalues, and if $\{(\varphi_i,\psi_i)\}_{i=1}^N$ is a cyclically monotone set, then every measure $\mu_t$ on the geodesic between $\mu_{\Phi}$ and $\mu_{\Psi}$ has support which spans $\Rd$.
\end{proposition}

\begin{proof}
Each measure on the geodesic $\mu_t$ will be supported on a new set of vectors, namely \\ ${\{(1-t)\varphi_i+t\psi_{\sigma(i)}\}_{i=1}^N,}$ and will be a probabilistic frame provided this set of vectors spans $\Rd$. Equivalently, $\mu_t$ will be a probabilistic frame if the probabilistic frame operator $S_{\mu_t}$ is positive definite. Let $P_{\sigma}$  be the $N\times N$ permutation matrix corresponding to $\sigma\in \Pi_N$, where now $\sigma$ is the optimal permutation for the Wasserstein distance.  Let $\Psi_{\sigma}=P_{\sigma}\Psi$.  A quick calculation shows:
$$S_{\mu_t}=\frac{1}{N}\left((1-t)\Phi^{\top}+t\Psi_{\sigma}^{\top}\right)\left((1-t)\Phi+t\Psi_{\sigma}\right).$$
\noindent $\Psi$ and $\Psi_{\sigma}$ have rank $d$, and to show that $S_{\mu_t}$ is positive definite, it remains to prove that every matrix in the set  $h(\Phi,\Psi_{\sigma}):=\{(1-t)\Phi+t\Psi_{\sigma}\}_{t\in [0,1]}$ has rank $d$.  By Lemma \ref{szulc_lemma}, a sufficient condition for this to be true is that $\Psi_{\sigma}^{\dagger}\Phi$ be positive semi-definite.  Finally, we note that if $\{(\varphi_i,\psi_i)\}_{i=1}^N$ is a cyclically monotone set, then $P_{\sigma} = I$, the identity, is an optimal permutation, and then $\Psi_{\sigma}^{\dagger}\Phi=\Psi^{\dagger}\Phi$ is positive definite by assumption.
\end{proof}

\begin{proof}{\bf Proof of Theorem~\ref{first_geod_thm}}

 With Lemma \ref{path_bound} showing that measures on the geodesic have finite second moment and Proposition \ref{proposition_DP} showing that the support of these measures spans $\Rd,$ Theorem \ref{first_geod_thm} is now proved. 
 \end{proof}

  Certain dual frame pairs immediately satisfy the conditions laid out in Theorem \ref{first_geod_thm}.
\begin{proposition}\label{can_dual}
If $\{\varphi_i\}_{i=1}^N$ is the canonical dual frame to $\{\psi_i\}_{i=1}^N$, then $\{(\varphi_i,\psi_i)\}_{i=1}^N$ is cyclically monotone.
\end{proposition}
\begin{proof}
Let $S=\Psi^{\top}\Psi$.  Then suppose that $\Phi^{\top}=S^{-1}\Psi^{\top}$. For any permutation $\sigma\in \Pi_N$, let $P_\sigma$ denote the matrix such that for $$\forall x= [x_1\dots x_N]^{\top}\in\R^N,\quad P_{\sigma}x=[x_{\sigma(1)}\dots x_{\sigma(N)}]^{\top}.$$

Then, 

\begin{align*}
\sum_{i=1}^N \ip{\varphi_i}{\psi_i-\psi_{\sigma(i)}}	&=	\sum_{i=1}^N \ip{S^{-1}\psi_i}{\psi_i-\psi_{\sigma(i)}} \\
													&=	\sum_{i=1}^N  (\psi_i-\psi_{\sigma(i)})^{\top}S^{-1}\psi_i\\
													&=	\Tr((\Psi-P_{\sigma}\Psi)S^{-1}\Psi^{\top})\\
													&=	\Tr((I_N-P_{\sigma})\Psi S^{-1}\Psi^{\top})\\
%													&=	\Tr((I_N-P_{\sigma})I_N^d)\\
													&\geq 0
\end{align*}
Here we use the fact that $\Psi S^{-1}\Psi^{\top}=I_N^d,$ the $N\times N$ diagonal matrix with $d$ leading ones on the diagonal and zeros else, because $S^{-1}\Psi^{\top}$ is the Moore-Penrose pseudoinverse of $\Psi$.  Therefore, the identity is an optimal permutation, i.e., the set $\{(\varphi_i,\psi_i)\}_{i=1}^N$ is cyclically monotone.
\end{proof}

\begin{proposition}\label{any_dual} Let  $\{\beta_i\}_{i=1}^N\subset\Rd$ be such  that $\{(\beta_i,\psi_i)\}_{i=d+1}^N$ is cyclically monotone. Then use $\{\beta_i\}_{i=1}^N$ to define $\{\varphi_i\}_{i=1}^N,$ one of the dual frames to $\{\psi_i\}_{i=1}^N$ as given in \eqref{christensen_duals}.  Then $\{(\varphi_i,\psi_i)\}_{i=1}^N$ is cyclically monotone.
\end{proposition}
\begin{proof}
Take $\{\varphi_i\}_{i=1}^N$ to be a dual of the form given in Equation \eqref{christensen_duals}.  Let $W$ be the matrix whose rows are the $\{\beta_i\}_{i=1}^N$.  Then, noting that $\Phi^{\top}=(S^{-1}\Psi^{\top} + W^{\top}(I_N -\Psi S^{-1}\Psi^{\top}))$,
\begin{align*}
\sum_{i=1}^N \ip{\psi_i-\psi_{\sigma(i)}}{\varphi_i}	&=	\Tr((I_N-P_{\sigma})\Psi\Phi^{\top}) \\
													&=	\Tr((I_N-P_{\sigma})\Psi(S^{-1}\Psi^{\top} + W^{\top}(I_N -\Psi S^{-1}\Psi^{\top})))\\
													&=	\Tr((I_N-P_{\sigma})I_N^d + (I_N-P_{\sigma})\Psi W^{\top}(I_N -I_N^d))\\
													&=	\Tr((I_N-P_{\sigma})I_N^d)+\sum_{i=d+1}^N\ip{\psi_i-\psi_{\sigma(i)}}{\beta_i}\\
													&\geq 0
\end{align*}
\noindent Therefore, under these conditions, $\{(\varphi_i,\psi_i)\}_{i=1}^N$ is cyclically monotone.
\end{proof}

\begin{proposition}\label{dual_discrete} If $\{\varphi_i\}_{i=1}^N$ is the canonical dual frame to $\{\psi_i\}_{i=1}^N$, or if $\{\varphi_i\}_{i=1}^N$ is a dual frame to $\{\psi_i\}_{i=1}^N$ of the form given in \eqref{christensen_duals}, with the $\{h_i\}_{i=1}^N$ ordered so that $\{(h_i,\psi_i)\}_{i=d+1}^N$ is cyclically monotone, then $\Psi_{\sigma}^{\dagger}\Phi$ is positive definite, where $\sigma$ is the optimal permutation for the Wasserstein distance. Consequently, any path on the geodesic joining $\{\psi_i\}_{i=1}^N$ and $\{\varphi_i\}_{i=1}^N$ is a probabilistic frame. 
\end{proposition}
\begin{proof}
By definition,
%\begin{align*}
$$\Psi_{\sigma}^{\dagger}	=	(P_{\sigma}\Psi)^{\dagger}
						=  (\Psi^{\top}P_{\sigma}^{\top}P_{\sigma}\Psi)^{-1}\Psi^{\top}P_{\sigma}^{\top}
						=  (\Psi^{\top}\Psi)^{-1}\Psi^{\top}P_{\sigma}^{\top}$$
%\end{align*}
This is a permutation of the matrix whose columns are canonically dual to the rows of $\Psi_{\sigma}$.  If $\{\varphi_i\}_{i=1}^N$ is any dual of $\{\psi_i\}_{i=1}^N$, then $\Psi^{\top}\Phi=I_d.$ Therefore, if $\sigma$ is the identity, then $\Psi_{\sigma}^{\dagger}\Phi=(\Psi^{\top}\Psi)^{-1}\Psi^{\top}\Phi=(\Psi^{\top}\Psi)^{-1}$, which is positive definite.  It remains to show that the optimal permutation is the identity.  But this is clear: Proposition \ref{can_dual} shows that if $\{\varphi_i\}_{i=1}^N$ is the canonical dual to $\{\psi_i\}_{i=1}^N$, then $\{(\varphi_i,\psi_i)\}_{i=1}^N$ is cyclically monotone, and Proposition \ref{any_dual} shows that if $\{\varphi_i\}_{i=1}^N$ is any dual to $\{\psi_i\}_{i=1}^N$ which meets the above condition, then $\{(\varphi_i,\psi_i)\}_{i=1}^N$ is cyclically monotone.
\end{proof}

There are other frame and dual-frame pairs which can easily be shown to meet the above conditions.  Consider the finite sequences $\{\varphi_i\}_{i=1}^N\subset \Rd$ and $\{\psi_i\}_{i=1}^N\subset \Rd$ with respective analysis operators $\Phi$ and $\Psi$.  Then the finite sequences are disjoint if $\Phi(\Rd) \bigcap \Psi(\Rd)=\{0\}$.  
%They are orthogonal if $\Phi^*\Psi=0$. 

\begin{proposition}\label{disjoint_prop}
If $\{\varphi_i\}_{i=1}^N$ and $\{\psi_i\}_{i=1}^N$ are disjoint frames for $\Rd$, associated canonically with the probabilistic frames $\mu_{\Phi}$ and $\mu_{\Psi}$, then every measure on the geodesic between $\mu_{\Phi}$ and $\mu_{\Psi}$ is a probabilistic frame.
\end{proposition}
\begin{proof}
Given $v\in\Rd$, consider:
\begin{align*}
\sum\limits_{i=1}^N\ip{v}{(1-t)\varphi_i+t\psi_i}^2	&=	\sum\limits_{i=1}^N\ip{v}{(1-t)\Phi^{\top}e_i+t\Psi^{\top}e_i}^2\\
											&=	\sum\limits_{i=1}^N\ip{(1-t)\Phi v+t\Psi v}{e_i}^2\\
											&=	\norm{(1-t)\Phi v+t\Psi v}_{\mathbf{R}^N}^2\\
											&\geq	C[(1-t)^2\norm{\Phi v}^2 +t^2\norm{\Psi v}^2]
\end{align*}
for some $C>0,$ since the frames are disjoint.  Since the two sequences in question are finite frames, choosing the minimum of the two lower frame bounds, say $A_0$, the last quantity can be bounded below by $(1-2t+2t^2)C\cdot A_0 \norm{v}^2$, yielding the result.
\end{proof}

\noindent Finally, in the following result control of the distance between the elements of a one frame and those of the canonical dual of the other by a coherence-like quantity guarantees the frame properties for the frames on the geodesic.
\begin{proposition}
Let $\{\psi_i\}_{i=1}^N$ be a dual frame to a frame $\{\varphi_i\}_{i=1}^N\subset S^{d-1}$.  Let $S_{\Phi}$ denote the frame operator.  For each $i$, let $z_i=\psi_i-S_{\Phi}^{-1}\varphi_i$, and let $a:=\min_{i\neq j}\ip{\varphi_i}{S_{\Phi}^{-1}(\varphi_i-\varphi_j)}$.  If $\max_j \norm{z_j}\leq\frac{a}{N}$, then the optimal $\sigma$ for the mass transport problem is the identity.
\end{proposition}
\begin{proof}

First, we note that $a\geq 0.$  If $a=0,$ then our hypothesis guarantees that $\norm{z_i}=\norm{\psi_i-S_{\Phi}^{-1}\varphi_i}=0$ for all $i$, so that $\Psi$ is the canonical dual to $\Phi,$ and in this case our result holds by Proposition \ref{can_dual}. Therefore, it only remains to consider the case when $a>0.$

%DISCUSSION
%
%By CBS for SPD matrices, given an SPD matrix $A$, and $x,y\in\Rd$, $$(\ip{x}{Ax}-\ip{x}{Ay})(\ip{y}{Ay}+\ip{x}{Ay})=\ip{x}{Ax}\ip{y}{Ay}-\ip{x}{Ay}^2\geq 0.$$  Therefore, if $(\ip{x}{Ax}-\ip{x}{Ay})<0$, then $(\ip{y}{Ay}+\ip{x}{Ay}) \leq 0.$ 
%But in this case, we can add the two expressions and see that $$(\ip{x}{Ax}-\ip{x}{Ay})+(\ip{y}{Ay}+\ip{x}{Ay}) = \ip{x}{Ax}+\ip{y}{Ay}<0,$$ which is a contradiction to $A$ being SPD.  Therefore, $(\ip{x}{Ax}-\ip{x}{Ay})\geq 0.$
%
%END DISCUSSION
For all $u,v\in\Rd$, $\sum\limits_{i=1}^N\ip{u}{z_i}\ip{v}{\varphi_i}=0$.  Then given $\sigma\in S_N$, let $n_{\sigma}$ be the number of elements not fixed by $\sigma$.  Then if $\sigma$ is the identity, $n_{\sigma}=0$ and
$$\sum\limits_{i=1}^N\ip{\varphi_i}{S_{\Phi}^{-1}x_{\sigma(i)}+z_{\sigma(i)}} = Tr(\Psi^{\top}P_{\sigma}\Phi)	= d$$
If $\sigma$ is not the identity, then
\begin{align*}
\sum\limits_{i=1}^N\ip{\varphi_i}{S_{\Phi}^{-1}\varphi_{\sigma(i)}+z_{\sigma(i)}} &  =\sum\limits_{i\neq\sigma(i)}\ip{\varphi_i}{S_{\Phi}^{-1}(\varphi_{\sigma(i)}-\varphi_i) + 					z_{\sigma(i)}-z_i} + d\\
																 &	\leq d-an_{\sigma}+ \sum_{i\neq\sigma(i)}\ip{\varphi_i}{z_{\sigma(i)}-z_i}\\											
													 &	\leq d-(1-\frac{2}{N})an_{\sigma}			 
\end{align*}
Since, given the hypothesis, for all $i,j$, $\ip{\varphi_i}{z_j}\leq\norm{\varphi_i}\norm{z_j}=\norm{z_j}\leq\frac{a}{N}.$  
Thus $\Tr(\Psi^{\top}P_{\sigma}\Phi)\leq d-(1-\frac{2}{N})an_{\sigma}\leq d = Tr(\Psi^{\top}\Phi)$ for all $\sigma,$ and it follows that the identity is the optimal transport map for the Wasserstein metric.

\end{proof}

\subsection{Absolutely Continuous Probabilistic Frames}

The question of the nature of the optimal transport plan for the 2-Wasserstein distance is simpler for absolutely continuous measures.  From \cite[Theorem 6.2.10 and Proposition 6.2.13]{AGS2005}, which gather together a long list of characteristics, two key facts about this plan can be extracted, which are collected in the following lemma. 
\begin{lemma}\cite[Chapter 6.2.3]{AGS2005}\label{r_monotone_lemma}
If $\mu_0$ and $\mu_1$ are absolutely continuous probability measures in $P_2(\Rd)$, then there exists a unique optimal transport plan for the 2-Wasserstein distance which is induced by a transport map $r$.  This transport map is defined (and injective) $\mu_0$-a.e.  Indeed, there exists a $\mu_0$-negligible set $N\subset\Rd$ such that $\ip{r(x_1)-r(x_2)}{x_1-x_2}>0$ for all $x_1, x_2 \in \Rd\setminus N$.
\end{lemma}
\noindent Then we have the following result for absolutely continuous probabilistic frames:
\begin{proposition}\label{linearAC_case}
If $\mu_0$ and $\mu_1$ are absolutely continuous (with respect to Lebesgue measure) probabilistic frames for which there exists a linear, positive semi-definite deterministic coupling which minimizes the Wasserstein distance, then all measures on the geodesic between these frames have support which spans $\Rd$ and will therefore be probabilistic frames.
\end{proposition}
\begin{proof}
Given the assumptions, let $r(x)$ denote the linear transformation which induces the coupling $\mu_1=r_{\hash}\mu_0$. 
Defining $h_t(x)=(1-t)x+tr(x)$ $\mu_0$-a.e., the geodesic measure is given by
\begin{equation}\label{ac_geodesic_meas}
\mu_t := {h_t}_{\hash}\mu_0.
\end{equation}
\noindent Then $S_{\mu_t}=\int_{\Rd}h_t(x)h_t(x)^{\top}d\mu_0(x)$.  If $r(x)=Ax$ for some $A \in \mathbf{A}^{d\times d}$, then:
\begin{align*}
S_{\mu_t}	&= \int_{\Rd}((1-t)Ix+tAx)((1-t)Ix+tAx)^{\top}d\mu_0(x)\\
			&= ((1-t)I+tA)S_{\mu_0}((1-t)I+tA)^{\top}
\end{align*}  
Since $A$ must be nonsingular--recall that $S_{\mu_1}=AS_{\mu_0}A^{\top}$, which is certainly of rank $d$--by Lemma \ref{szulc_lemma}, $(1-t)I+tA$ will also nonsingular for all $t\in [0,1]$ provided that $A$ has no negative eigenvalues, as we assumed.
\end{proof}

\begin{example}
An example in which the assumptions of the above proposition hold is the case of nondegenerate Gaussian measures on $\Rd$.  Let $\mu_0$ and $\mu_1$ be zero-mean Gaussians. Let $r(x)=S_{\mu_1}^{\half}S_{\mu_0}^{-\half}x$.  According to a result in \cite{DowsonLandau}, if $X$ and $Y$ are two zero-mean random vectors with covariances $\Sigma_X$ and $\Sigma_Y$, respectively, then a lower bound for $E(\norm{X-Y}^2)$ is $\Tr[\Sigma_X +\Sigma_Y - 2(\Sigma_X \Sigma_Y)^\half]$, and the bound is attained, for nonsingular $\Sigma_X$, when $Y=\Sigma_X^{-\half}\Sigma_Y^{\half}X$, so that the coupling $r$ is an optimal positive definite linear deterministic coupling of $\mu_0$ and $\mu_1$. 
\end{example}

\noindent Now, given absolutely continuous probabilistic frames $\mu, \nu$ for $\Rd$, take $r(x)$ to be the optimal transport map pushing $\mu$ to $\nu$ guaranteed by Lemma \ref{r_monotone_lemma}. Define $$h_t(x)=(1-t)x + tr(x) \quad\textrm{for } t \in [0,1];$$ then $S_{\mu_t} = \int h_t(x) \otimes h_t(x) d\mu(x)$, with $\mu_t=(h_t)_{\hash}\mu$. Then we can state the following:

\begin{proposition}\label{ht_injective_prop}
Given two such probabilistic frames, there exists a set $N$ with $\mu(N)=0$ such that $h_t$ is injective for all $t\in[0,1]$ on $\supp(\mu)\setminus N$.
\end{proposition}
\begin{proof}
Given $x, y \in \supp(\mu)\setminus N$, with N as defined in Lemma \ref{r_monotone_lemma}, suppose $h_t(x)=h_t(y)$ for some $t\in[0,1]$.  Then, since:
\begin{align*}
0		& = \ip{h_t(x)-h_t(y)}{x-y} \\
		&= \ip{(1-t)(x-y) +t(r(x)-r(y))}{x-y}\\
		& = (1-t)\norm{x-y}^2 + t\ip{r(x)-r(y)}{x-y}
\end{align*}						
it follows that $$\ip{r(x)-r(y)}{x-y} = \frac{t-1}{t}\norm{x-y}^2.$$
This implies that $\ip{r(x)-r(y)}{x-y} \leq 0.$  However, from the proposition above, we also know that $\ip{r(x)-r(y)}{x-y} \geq 0$.  Therefore $\norm{x-y}=0$, and $h_t$ is injective on $\supp(\mu)\setminus N$.
\end{proof}

This injectivity claim is crucial for the main result of this section:

\begin{theorem}\label{cty_geodesic}
Let $\mu,\nu\in P_2^r(\Rd)$, and let $r$ be the unique optimal transport map for the $2-Wasserstein$ distance. Let $N$ be the set of measure zero define in Proposition \ref{ht_injective_prop}. If $r$ is continuous, and if $\supp(\mu)\setminus N$ contains an open set, then every geodesic measure $\mu_t$ is a probabilistic frame.
\end{theorem}
\begin{proof}
Since $r$ is continuous and, by Proposition \ref{ht_injective_prop}, injective outside a set $N$ of measure zero, so is $h_t$ for each $t$.  Let $x_0\in\supp(\mu)\setminus N$.  First, we show that for any $\epsilon>0$, $h_t^{-1}(B_{\epsilon}(h_t(x_0)))$ contains an open set containing $x_0$.

Since $h_t$ is continuous at any such $x_0$, given $\epsilon>0$, there exists $\delta>0$ such that $\forall x\in B_{\delta}(x_0)$, 
$\norm{h_t(x)-h_t(x_0)}< \epsilon$.  Hence for any $x\in B_{\delta}(x_0)$, $x\in h_t^{-1}(B_{\epsilon}(h_t(x_0)))$--i.e., $B_{\delta}(x_0)\subset h_t^{-1}(B_{\epsilon}(h_t(x_0)))$.

Then $\forall x_0\in\supp(\mu)\setminus N$, consider $B_{\frac{1}{k}}(h_t(x_0))$:
\begin{align*}
\mu_t(B_{\frac{1}{k}}(h_t(x_0)))	& =	\int\indicator{B_{\frac{1}{k}}(h_t(x_0))}(h_t(y))d\mu(y)\\
								& =	\int\indicator{h_t^{-1}(B_{\frac{1}{k}}(h_t(x_0)))}(y)d\mu(y)\\
								& = \mu(h_t^{-1}(B_{\frac{1}{k}}(h_t(x_0))))\\
								& >	0
\end{align*}
\noindent where the last inequality holds since $x_0\in\supp(\mu)$ and, as shown above, $h_t^{-1}(B_{\frac{1}{k}}(h_t(x0))))$ contains an open set containing $x_0$.  Thus, we have shown that for any $k\in\N$, the open ball of radius $\frac{1}{k}$ around $h_t(x_0)$ has positive $\mu_t$-measure, and therefore $h_t(x_0)$ lies in $\supp(\mu_t)$.  Thus $h_t(\supp(\mu)\setminus N)\subset\supp(\mu_t)$.

Therefore, since $h_t$ is injective by Proposition \ref{ht_injective_prop} above and continuous on $\supp(\mu)\setminus N$ and by assumption, there exists open set $U\subset\supp(\mu)\setminus N$, by invariance of domain, $h_t(U)\subset\supp(\mu_t)$ is open, we conclude that $h_{t\hash}\mu$ has support which spans $\Rd$.
\end{proof}

\noindent The question of when $r$ is continuous is the subject of ongoing research.  One example is when $\mu$ and $\nu$ are supported on a bounded convex subset of $\Rd$ \cite{FigalliDePhilippis2013}.

\section*{Acknowledgment}
Clare Wickman would like to thank the Norbert Wiener Center
for Harmonic Analysis and Applications for its support during
this research. 
Kasso Okoudjou was partially supported by a grant from the Simons Foundation ($\# 319197$ to Kasso Okoudjou), and ARO grant W911NF1610008. 

\bibliographystyle{amsplain}
\bibliography{Bibliography} 

\end{document}